\documentclass{amsproc}
\usepackage{tikz}

\newtheorem{theorem}{Theorem}[section]
\newtheorem{lemma}[theorem]{Lemma}
\newtheorem{corollary}[theorem]{Corollary}

\theoremstyle{definition}

\newtheorem{example}[theorem]{Example}

\theoremstyle{remark}
\newtheorem{remark}[theorem]{Remark}

\numberwithin{equation}{section}

\begin{document}

\title{The $n$-total graph of a commutative ring}


\author{Djamila AitElhadi}
\address{Department of Mathematics and Statistics, American University of Sharjah, P.O. Box 26666, Sharjah, United Arab Emirates}
\email{g00098182@alumni.aus.edu}


\author{Ayman Badawi}
\address{Department of Mathematics and Statistics, American University of Sharjah, P.O. Box 26666, Sharjah, United Arab Emirates}

\email{abadawi@aus.edu}

\subjclass[2020]{Primary 13A15; Secondary 13B99, 05C99}

\keywords{total graph, zerodivisor graph, generalized total graph, annihilator graph, zerodivisor, connected, diameter, girth}

\begin{abstract}
Let $R$ be a commutative ring with $1\not = 0$, $Z(R)$ be the set of all zero-divisors of $R$, and $n \geq 1$. This paper introduces the $n$-total graph of a commutative ring $R$. The $n$-total graph of a commutative ring $R$, denoted by $n-T(R)$, is an undirected simple graph with vertex set $R$, such that two vertices $x, y$ in $R$ are connected by an edge if $x^n + y^n$ in $Z(R)$. Note that if $n =1$, then the $1$-total graph of $R$ is the total graph of $R$ in the sense of Anderson-Badawi's paper on the total graph of a commutative ring. In this paper, we study some graph properties and theoretical ring structure.
\end{abstract}

\maketitle

\section{Introduction}
Graphs represent relations defined by taking ordered pairs from a set of vertices. Such ordered pairs are called edges. Graphs are very beneficial in physical, biological, and social applications. They are also used in computer scientific contexts like networks and data structure representations. In addition, Graphs are utilized in other mathematical fields, such as knot theory, where each type of knot can be associated with a graph.

Graphs over rings were first seriously researched during the late nineties of the last century. While the discussion was mainly about graphs over rings, there has been research on graphs over groups, semi-rings, modules over rings, etc. For a recent reference on graphs from rings, see the book \cite{5} by Anderson, Asir, Badawi, and Chelvam. In the paper \cite{17} published in 1988 titled 'Colouring of Commutative rings' by Istvan Beck, the idea of the colouring of commutative rings was presented. This idea established a connection between graph theory and commutative ring theory that is evident in many studies such as  \cite{4} and \cite{20}. The paper establishes the graph of $R$ by defining the relation $x$ is adjacent to $y$ if and only if $xy=0$ in a commutative $R$. The study is mainly concerned with characterizing and discussing finitely colourable rings and presenting certain conditions for which a ring $R$ is finitely colourable.  Anderson and Livingston introduce the zero divisor graph $\Gamma(R)$ of a commutative ring $R$  in \cite{11} by defining the vertices of the graph as $V=Z(R)-\{0\}$ and defining the edges by the relation $x \in V$ is adjacent to $y \in V$ if and only if $xy=0$. The zero divisor graph was investigated in several studies, such as \cite{6}, \cite{10}, and \cite{19}.  Another body of work that is well-studied in the field was the work of David F. Anderson and Ayman Badawi in \cite{7} on the total graph of a commutative ring $R$ defined by taking the vertices $V=\{x \in R\}$ and defining the edges by $x$ is adjacent to $y$ if and only if $x+y\in Z(R)$.

Let $R$ be a commutative ring with $1\not = 0$ and let $n \geq 1$ be an integer. In this paper, we define the $n$-total graph of $R$ by taking the vertices $V= R$ and defining the edges by $x \in V$ is adjacent to $y \in V$ if and only if $x^n+y^n\in Z(R)$. We denote the $n$-total graph of $R$ by $n-T(R)$. The literature on graphs from rings is rich. We cannot state them all; for example see \cite{1}--\cite{20}. A complete list of references up to 2021 that might interest the reader is \cite{5}.

We recall some definitions that are needed in this paper.   Let $G$ be a (simple) graph. We say that $G$ is {\it connected} if there is a path between any two distinct vertices of $G$. At the other extreme, we say that $G$ is {\it totally disconnected} if no two vertices of $G$ are adjacent (i.e., no vertices of $G$ are connected by one edge). For vertices $x$ and $y$ of $G$, the {\it distance} between $x$ and $y$, denoted by $d(x, y)$, is defined to be the length of the shortest path from $x$ to $y$ ($d(x, x) = 0$ and if there is no path between $x$ and $y$, then $d(x, y) = \infty$). The {\it diameter} of $G$ is $diam(G) = sup\{d(x, y) \mid$  $x$ and $y$ are vertices of $G\}$. The {\it girth} of $G$, denoted by $gr(G)$, is the length of a shortest cycle in $G$ ($gr(G) = \infty$ if $G$ contains no cycles). We denote the complete graph on $n$ vertices by $K_n$ and
the complete bipartite graph on $m$ and $n$ vertices by $K_{m, n}$ (we allow $m$ and $n$ to be infinite cardinals). We say that a
(induced) subgraph $G_1$ is a {\it component} of a graph $G$ if $G_1$ is connected and no vertex of $G_1$ is adjacent (in $G$) to any vertex not in $G_1$. If $G$ is the union of $m$ components, we say $G$ is a decomposition of $m$ components, and we write $G =\bigoplus_{i = 1}^m G_i$ and if $G_1 = \cdots = G_m$, then we write $G = \bigoplus_{i = 1}^m G_1$.

Throughout this paper, all rings are commutative with $1\not = 0$. Let $R$ be a commutative ring. Then $Z(R)$ denotes its set of all zero-divisor, $Nil(R)$ denotes its ideal of nilpotent elements, $Reg(R)$ denotes its set of non-zero-divisors (i.e., $Reg(R) = R \setminus Z(R)$), and $U(R)$ denotes its group of units. For $A \subseteq R$, let $A^* = A - \{0\}$. We say that $R$ is {\it reduced} if $Nil(R) = \{0\}$. As usual, $\mathbb{Z}, \mathbb{Z}_n$, and $F_q$ will denote the integers, integers modulo $n$, and the finite field with $q$ elements, respectively.

\section{Results and Analysis}

Let $n \geq2$. First, we will show that the $n$-total graph of a ring differs from the total graph of a ring, as in the Anderson-Badawi paper. We have the following examples.

\begin{example}
	The total graph of $\mathbb{Z}_{49}$ is a decomposition of four components, $K_7\oplus K_{7,7} \oplus K_{7,7} \oplus K_{7,7}$. Meanwhile, the $3$-total graph of $\mathbb{Z}_{49}$ is a decomposition of two components, $K_7\oplus K_{21,21}$.
\end{example}
\begin{example}
	The total graph of $\mathbb{Z}_7$ has $4$ components, while the $3-T(\mathbb{Z}_7)$ has $2$ components, and the $2-T(\mathbb{Z}_7)$ has $7$ components
\end{example}
\begin{example}
	The total graph of $\mathbb{Z}_{13}$ has $7$ components, the $6$-total graph of $\mathbb{Z}_{13}$ has $2$ components, and the $2$ total graph of $\mathbb{Z}_{13}$ has $4$ components
\end{example}
\subsection{The Case Where $Z(R)$ is an Ideal}\hfill\par
Let $R$ be a commutative ring such that $Z(R)$ is an ideal of $R$. Then  $Z(R)$ is a prime ideal of $R$; therefore, $R/Z(R)$ is an integral domain. Moreover, when $R$ is a finite commutative ring, $R/Z(R)$ becomes a field as $Z(R)=Nil(R)$ is the only maximal ideal of the ring. Trivially, the induced subgraph $n-T(Z(R))$ is a complete component of the $n-T(R)$. Thus, the main focus of this section will be on $n-T(Reg(R))$. We will generalize the results seen in \cite{7} for finite commutative rings where $Z(R)$ is an ideal of $R$. The following two lemmas are known; we omit their proofs.
\begin{lemma}\label{L1}
Let $F_q$ be a finite field. Then, $U(F_q) = F_q^*$  is a cyclic group. Furthermore, for each positive integer $m \geq 1$, if the equation $x^m = a$ has a solution in $F_q^*$ for some $a\in F_q^*$, then there are precisely $gcd(m, |F_q^*|)$ distinct solutions in $F_q^*$.
\end{lemma}

\begin{lemma}\label{L2}
	Let $F_q$ be a finite field, $a \in F_q^*$, and $m\geq 1 $ be an integer. Then $x^m=a$ has a solution in $F_q$ if and only if $a \in S_m = \{b^m \mid b \in F_q^*\}$. Furthermore, $S_m$ is a cyclic subgroup of $F_q^*$ of order $\frac{|F_q^*|}{gcd(m, |F_q^*|)}$.
\end{lemma}

The following theorem generalizes \cite[Theorem 2.2 (1)]{7} for finite commutative rings.
\begin{theorem}\label{T1}
	Let $R$ be a finite commutative ring such that $Z(R)$ is an ideal of $R$ and $2 \in Z(R)$. Let $n \geq 1$ be an integer, $|Z(R)|=\alpha$, $|R/Z(R)|=\beta$, $\gamma = \alpha\cdot gcd(n, \beta - 1)$, and $d = \frac{\beta - 1}{gcd(n, \beta - 1)}$. Then the $n-T(R)= K_\alpha\oplus\bigoplus_{i =1}^{d}K_{\gamma}$. Note that $|R| = \alpha\beta$.
\end{theorem}
\begin{proof}
	Let $n \geq 1$ be an integer. Assume that $2 \in Z(R)$, and let $x \in Reg(R)$ (i.e., $x \in U(R)$). Since $2 \in Z(R)$, $(x + z_1)^n + (x + z_2)^n = 2x^n + z_3\in Z(R)$ for some $z_3\in Z(R)$ for all $z_1, z_2 \in Z(R)$. Hence, each coset $x + Z(R)$ elements form a $K_\alpha$ subgraph of the $n-T(R)$. Since $F = |R/Z(R)|$ is a finite field, for each $a \in S_n(F) = \{b^n \mid b \in F^*\}$, the set $C_n(a) = \{w \in F^* \mid w^n = a \}$ has exactly $gcd(n, \beta - 1)$ cosets of $Z(R)$ by Lemma \ref{L1} and Lemma \ref{L2}. Let $D_a = \cup_{v \in C_n(a)} v$. Then $D_a$ is a subset of $R$, $|D_a| = \gamma$, and every two vertices in $D_a$ are adjacent since $2 \in Z(R)$. Hence, $D_a$ is a $K_\gamma$ component of the $n-T(R)$. Since $|S_n(F)| = d$ by Lemma \ref{L2} and $Z(R)$ forms a $K_\alpha$ component of the $n-T(R)$, we conclude that the $n-T(R)= K_\alpha \oplus \bigoplus_{i = 1}^dK_{\gamma}$.

\end{proof}

\begin{example}
	Let $p=2$, $m=2^i$, $c=2^{i-1}$, and let $n\geq 1$ be an integer. The $n-T(\mathbb{Z}_m) = K_c  \oplus K_c$ where the $n-T(Reg(\mathbb{Z}_m)) = K_c$.
\end{example}
\begin{example}
	Let $R=\mathbb{Z}_8$. We have $2\in Z(R)$, $|Z(R)|=4$ and $|R/Z(R)|=2$. The $2$-total graph of $\mathbb{Z}_8$ is $K_{4}\oplus K_4$. Figure 1 shows the $n$-total graph of $\mathbb{Z}_8$ for every integer $n \geq 1$.
\end{example}
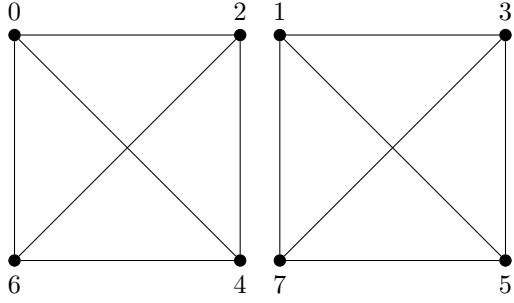
\begin{figure}[!h]
	
	\centering
	\begin{tikzpicture}
		\tikzset{enclosed/.style={draw, circle, inner sep=0pt, minimum size=.15cm, fill=black}}
		
		\node[enclosed, label={above, yshift=0cm: 0}] (0) at (0,3) {};
		\node[enclosed, label={above, xshift=0cm: 2}] (2) at (3,3) {};
		\node[enclosed, label={below, yshift=0cm: 4}] (4) at (3,0) {};
		\node[enclosed, label={below, xshift=0cm: 6}] (6) at (0,0) {};
		
		\draw (0) -- (2) node[midway, above] (edge1) {} ;
		\draw (2) -- (4) node[midway, left] (edge2) {};
		\draw (4) -- (6) node[midway, right] (edge3) {};
		\draw (6) -- (0) node[midway, below] (edge4){};
		\draw (6) -- (2) node[midway, below] (edge4){};
		\draw (0) -- (4) node[midway, above] (edge1) {} ;
		
	\end{tikzpicture}
	\begin{tikzpicture}
		\tikzset{enclosed/.style={draw, circle, inner sep=0pt, minimum size=.15cm, fill=black}}
		
		\node[enclosed, label={above, yshift=0cm: 1}] (0) at (0,3) {};
		\node[enclosed, label={above, xshift=0cm: 3}] (2) at (3,3) {};
		\node[enclosed, label={below, yshift=0cm: 5}] (4) at (3,0) {};
		\node[enclosed, label={below, xshift=0cm: 7}] (6) at (0,0) {};
		
		\draw (0) -- (2) node[midway, above] (edge1) {} ;
		\draw (2) -- (4) node[midway, left] (edge2) {};
		\draw (4) -- (6) node[midway, right] (edge3) {};
		\draw (6) -- (0) node[midway, below] (edge4){};
		\draw (6) -- (2) node[midway, below] (edge4){};
		\draw (0) -- (4) node[midway, above] (edge1) {} ;
			\end{tikzpicture}
	
	\caption{$n-T(\mathbb{Z}_8)$ for every integer $n \geq 1$}
\end{figure}

The following theorem generalizes \cite[Theorem 2.2 (2)]{7} for finite commutative rings.
\begin{theorem}\label{T2}
	Let $n \geq 1$ be an integer and $R$ be a finite commutative ring such that $Z(R)$ is an ideal of $R$ and $2\not \in Z(R)$. Let $|Z(R)|=\alpha$, $|R/Z(R)|=\beta$,  $\gamma =\alpha\cdot gcd(n, \beta - 1)$, and $d=\frac{\beta-1}{gcd(n,\beta-1)}$. Then
	\begin{enumerate}
		\item   The $n-T(Reg(R))$ is totally disconnected if and only if  $d$ is odd.
		\item If $d$ is even, then the $n-T(Reg(R))= \bigoplus_{i = 1}^{d/2}K_{\gamma, \gamma}$, and hence the $n-T(R) = K_\alpha \oplus \bigoplus_{i = 1}^{d/2}K_{\gamma, \gamma}$.
	\end{enumerate}
\end{theorem}
\begin{proof}

	Let $n \geq 1$ be an integer. Note that $F = R/Z(R)$ is a finite field and $a^{\beta - 1} = 1 + Z(R)$ for every $a \in F^*$.
		
	(1). Suppose the $n-T(Reg(R))$ is totally disconnected. Assume that $d$ is even. Then $S_n(F) = \{b^n \mid b \in F^*\}$ is the unique cyclic subgroup of $F^*$ of order $d$ by Lemma \ref{L2}. Since $d$ is even, $-1 + Z(R) \in S_n(F)$. Since $2 \not \in Z(R)$, $1 + Z(R) \not = -1 + Z(R)$.  Hence there are $x = a + Z(R), y = b + Z(R) \in F^*$ for some $a, b \in Reg(R)$ such that $x^n + y^n = (a^n + b^n) + Z(R) = Z(R)$ in $F$. Thus $a^n + b^n \in Z(R)$; a contradiction since the $n-T(R)$ is totally disconnected.
				
		For the converse, suppose $d$ is odd. Note that $dn = k\cdot (\beta - 1)$ for some integer $k\geq 1$. Assume that the $n-T(Reg(R))$ is not totally disconnected. Then, there exist $x,y \in Reg(R)$ such that $x^n+y^n \in Z(R)$. Thus $x^n + Z(R) = - y^n + Z(R)$ in $F$. Hence $(x^n)^d + Z(R) = (-y^n)^d + Z(R)$ in $F$. Thus $x^{nd} + Z(R) = (-1^d\cdot y^{nd}) + Z(R)$ in $F$. Since $nd = k(\beta - 1)$ and $d$ is odd, we have $1 + Z(R) = -1 + Z(R)$ in $F$. Hence $2 \in Z(R)$, a contradiction. Thus the $n-T(R)$ is totally disconnected.
			
	(2). Let $a \in Reg(R)$. Since $2 \not \in Z(R)$, $a^n+a^n = 2a^n \not \in Z(R)$. Therefore, no two vertices in $a+Z(R)$ are connected in the $n-T(Reg(R))$. Since $S_n(F) = \{b^n \mid b \in F^*\}$ is the unique cyclic subgroup of $F^*$ of order $d$ by Lemma \ref{L2} and $d$ is even, we conclude that $-a \in S_n(F)$ for every $a\in S_n(F)$.  Since $F = |R/Z(R)|$ is a finite field, for each $a \in S_n(F) = \{b^n \mid b \in F^*\}$, the set $C_n(a) = \{w \in F^* \mid w^n = a \}$ has exactly $gcd(n, \beta - 1)$ cosets of $Z(R)$ by Lemma \ref{L1} and Lemma \ref{L2}. Let $A = \cup_{v \in C_n(a)} v $ and $B = \cup_{w \in C_n(-a)} w$. Then $A \cup B$ is a subset of $Reg(R)$, and $|A| = |B| = \gamma$. Note that every two vertices of $A$ are not adjacent, and every two vertices of $B$ are not adjacent. However, every vertex in $A$ is adjacent to every vertex in $B$.  Thus $A \cup B$ elements form a $K_{\gamma, \gamma}$ component of the $n-T(Reg(R))$. Thus the $n-T(Reg(R))= \bigoplus_{i = 1}^{d/2}K_{\gamma, \gamma}$. Since $Z(R)$ elements form a $K_\alpha$ component of the $n-T(R)$. We conclude that the $n-T(R) =K_\alpha \oplus \bigoplus_{i = 1}^{d/2}K_{\gamma, \gamma}$
\end{proof}

\begin{example} Let $n \geq 1$ be an integer, $p$ be an odd prime, $m = p^i$, $R = \mathbb{Z}_m$,  $\alpha = p^{i-1} = |Z(\mathbb{Z}_m)|$, $\beta = |R/Z(R)|$, $\gamma = gcd(n, \beta - 1)\alpha$, and $d = (\beta - 1)/gcd(n, \beta - 1) = (p-1)/gcd(n,p-1)$.
	\begin{enumerate}
		\item  The $n-T(Reg(\mathbb{Z}_m))$ is totally disconnected if and only if  $d$ is odd.
		\item If $d$ is even, then the $n-T(Reg(R))= \bigoplus_{i = 1}^{d/2}K_{\gamma, \gamma}$, and hence the $n-T(R) = K_\alpha \oplus \bigoplus_{i = 1}^{d/2}K_{\gamma, \gamma}$.
		\item Let $R = \mathbb{Z}_{169}$ and $n = 3$. Let $\alpha, \beta, \gamma$, and $d$ as in Theorem \ref{T2}. Then $\alpha = \beta = 13$, $\gamma = 39$, and $d = 4$. Then the $3-T(R) = K_{39, 39} \oplus K_{39, 39} \oplus K_{13}$.
	\end{enumerate}
	
\end{example}

\begin{example}
Let  $R = \mathbb{Z}_9$ and $n = 2$. We have $\alpha = |Z(R)|= 3$, $\beta = |R/Z(R)| = 3$, $\gamma = gcd(n, \beta - 1)\alpha = 3gcd(2, 2) = 6$, and $d = (\beta - 1)/gcd(n, \beta - 1) = 2/2 = 1$.  Then the  $2-T(Reg(R))$ is a totally disconnected graph with $6$ vertices. However, The total graph of $Reg(R) = 1-T(Reg(R)) = K_{3,3}$. Thus the $2-T(R) = K_{3,3} \oplus K_3$
\end{example}

\begin{corollary}
	Let $n \geq 1$ be an integer and $R$ be a finite commutative ring such that $2\not\in Z(R)$. Let $|Z(R)|=\alpha$, $|R/Z(R)|=\beta$, and $\gamma = \alpha \cdot gcd(n,\beta-1)=(\beta-1)/2$. Then the $n-T(Reg(R))$ is connected if and only if $gcd(n,\beta-1)=(\beta-1)/2$. Furthermore, if the $n-T(Reg(R))$ is connected, then the $n-T(R) = K_\alpha \oplus K_{\gamma, \gamma}$, where the $n-T(Reg(R)) = K_{\gamma, \gamma}$ is connected.
\end{corollary}

\begin{proof}
	Let $d = (\beta - 1)/gcd(n, \beta - 1)$.
Suppose that the $gcd(n,\beta-1)=(\beta-1)/2$. Then $d = 2$. Hence by Theorem \ref{T2}, the $n-T(Reg(R))= K_{\gamma, \gamma}$ is connected. Conversely, suppose the $n-T(Reg(R))$ is connected. Since the $n-T(Reg(R))=\bigoplus_{i = 1}^{d/2} K_{\gamma, \gamma}$ by Theorem \ref{T2}, we conclude that $d=2$. Thus the $ gcd(n,\beta-1)=(\beta-1)/2$.
\end{proof}

\begin{example}
Let $n \geq 1$ be an integer, $p$ be an odd prime, $m = p^i$, $R = \mathbb{Z}_m$,  $\alpha = p^{i-1} = |Z(\mathbb{Z}_m)|$, $\beta = |R/Z(R)| = p$, $\gamma = gcd(n, \beta - 1)\alpha = gcd(n, p-1)\alpha$, and $d = (\beta - 1)/gcd(n, \beta - 1) = (p-1)/gcd(n,p-1)$.  Then the $n-T(Reg(R))$ is connected if and only if $gcd(n,p-1)=(p-1)/2$. Furthermore, if the $n-T(Reg (R))$ is connected, then the $n-T(R) = K_\alpha  \oplus K_{\gamma, \gamma}$, where the $n-T(Reg(R)) = K_{\gamma, \gamma}$ is connected.
\end{example}

\begin{example}
The $2-T(Reg(\mathbb{Z}_{25}))$ is connected since $gcd(2,5-1)=2$. This is unlike $1-T(Reg(\mathbb{Z}_{25}))$, which is not connected and has two complete bipartite components.
\end{example}
The following theorem generalizes \cite[Theorem 2.5]{7} for finite commutative rings.
\begin{theorem}\label{T3}
Let $n \geq 1$ be an integer and $R$ be a finite commutative ring such that $Z(R)$ is an ideal of $R$. Then
\begin{enumerate}
	\item $diam(n-T(Reg(R)))=0,1,2,$ or $\infty$.
	\item $gr(n-T(Reg(R)))= 3,4,$ or $\infty$.
\end{enumerate}
\end{theorem}
\begin{proof}
\begin{enumerate}
	\item Suppose the $n-T(Reg(R))$ is connected. Then the $n-T(Reg(R))$ is a singleton, a complete graph, or a complete bipartite graph by Theorem \ref{T1} and Theorem \ref{T2}.  Therefore, $diam(n-T(R))\leq 2$.
	\item Suppose the  $n-T(Reg(R))$ has a cycle. Then the $n-T(Reg(R))$ is a complete graph or a disjoint union of complete bipartite graphs by Theorem \ref{T1} and Theorem \ref{T2}. Hence, if a component of the  $n-T(Reg(R))$ has more than two vertices, then $gr(n-T(Reg(R)) = 3$ or $4$. If each component of the $n-T(Reg(R))$ has two vertices or fewer, then  $gr(n-T(Reg(R)) = \infty$.
\end{enumerate}
\end{proof}
\begin{example}
	Let $n \geq 1$ be an integer. Then
\begin{enumerate}
	\item $diam(2-T(Reg(\mathbb{Z}_2)))=0$ and $gr(n-T(Reg(\mathbb{Z}_2)))=\infty$ as $Reg(\mathbb{Z}_2)=\{1\}$ is a singleton.
	\item $diam(2-T(Reg(\mathbb{Z}_8)))= 1$,  $gr(n-T(Reg(\mathbb{Z}_8)))= 3 $, and  the $gr(2-T(Reg(\mathbb{Z}_{25}))) = 4$.
\end{enumerate}

\end{example}

\subsection{The Case Where $Z(R)$ is not an Ideal}
\begin{theorem}\label{T3}
	Suppose $R$ is a commutative ring such that $Z(R)$ is not an ideal of $R$. Let $n \geq 1$ be an integer. Then the $n-T(R)$ is connected if and only if there exists a path from $0$ to $1$ in the $n-T(R)$.
\end{theorem}
\begin{proof}
	If the $n-T(R)$ is connected, then there is a path from $0$ to $1$. Conversely, suppose there exists a path from $0$ to $1$ in the  $n-T(R)$, say $0-w_1-w_2- \cdots  -w_q = 1$. Then for any $w_i- w_j$, $w_i^n+w_j^n \in Z(R)$. Let $v\in R^*$. Then $(vw_i)^n+(vw_j)^n=v^nw_i^n+v^nw_j^n = v^n(w_i^n+w_j^n)\in Z(R)$. Hence, $0- vw_1- \cdots - vw_q = v$ is a path from $0$ to $v$ for all $v$ in $R$, and therefore, there is a path between any two vertices in $R$. Therefore, the $n-T(R)$ is connected.
\end{proof}

The next result generalizes \cite[Theorem 3.1, (3)]{7} for odd integers.
\begin{theorem}\label{T4}
Let  $R$ be a commutative ring such that $Z(R)$ is not an ideal of $R$.	Let $n \geq 1$ be an odd integer. If the $n-T(Reg(R))$ is connected, then the  $n-T(R)$ is connected.
\end{theorem}
\begin{proof}
 Suppose the $n-T(Reg(R))$ is connected. Since $Z(R)$ is not an ideal of $R$, there are $w, z\in Z(R)$ such that $w - z=r\in Reg(R)$. Therefore, we have $(w-z)^n=r^n = \sum^{n}_{k=1}  {n \choose
		k} (-1)^k w^k z^{n-k}= \sum^{n-1}_{k=1}  {n \choose
		k}(-1)^k w^k z^{n-k}-z^n=wq-z^n$ for some $q\in R$. Therefore, we have $r^n+z^n=wq\in Z(R)$. Since $0, z \in Z(R)$, every vertex of $Z(R)$ is connected to $z$. Since $z$ is connected to $r\in Reg(R)$ and the $n-T(Reg(R))$ is connected, we conclude the $n-T(R)$ is connected.
\end{proof}
\begin{example}
	The $3-T(\mathbb{Z}_6)$ is connected, as shown in figure 3.  Note that the $3-T(Reg(\mathbb{Z}_6))$ is connected, as shown in Figure 2.
\end{example}

\begin{figure}[!h]
	\centering
	\begin{tikzpicture}
		\tikzset{enclosed/.style={draw, circle, inner sep=0pt, minimum size=.15cm, fill=black}}
		
		\node[enclosed, label={above, yshift=0cm: 1}] (1) at (5,0) {};
		
		\node[enclosed, label={above, yshift=0cm: 3}] (3) at (4,1) {};
		
		\node[enclosed, label={above, yshift=0cm: 5}] (5) at (3,0) {};
		
		
		\draw (3) -- (1) node[midway, above] (edge1){};
		\draw  (1) -- (5) node[midway, above] (edge1){};
		\draw (5)--(3) node[midway, above] (edge1){};
	\end{tikzpicture}
	\caption{$3-T(Reg(\mathbb{Z}_6))$}
	\label{fig:enter-label}
\end{figure}
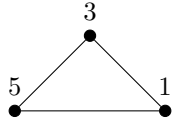
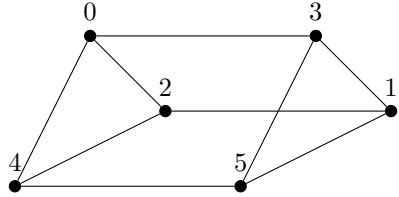
\begin{figure}[!h]
	\centering
	\begin{tikzpicture}
		\tikzset{enclosed/.style={draw, circle, inner sep=0pt, minimum size=.15cm, fill=black}}
		\node[enclosed, label={above, yshift=0cm: 0}] (0) at (1,2) {};
		\node[enclosed, label={above, yshift=0cm: 1}] (1) at (5,1) {};
		\node[enclosed, label={above, yshift=0cm: 2}] (2) at (2,1) {};
		\node[enclosed, label={above, yshift=0cm: 3}] (3) at (4,2) {};
		\node[enclosed, label={above, yshift=0cm: 4}] (4) at (0,0) {};
		\node[enclosed, label={above, yshift=0cm: 5}] (5) at (3,0) {};
		
		\draw (0) -- (3) node[midway, above] (edge1){};
		\draw (0) -- (2) node[midway, above] (edge1){};
		\draw (0) -- (4) node[midway, above] (edge1){};
		\draw (2) -- (4) node[midway, above] (edge1){};
		\draw (3) -- (1) node[midway, above] (edge1){};
		\draw  (1) -- (5) node[midway, above] (edge1){};
		\draw (5)--(3) node[midway, above] (edge1){};
		\draw (4)--(5) node[midway, above] (edge1){};
		\draw (1)--(2) node[midway, above] (edge1){};
	\end{tikzpicture}
	\caption{$3-T(\mathbb{Z}_6)$}
	\label{fig:enter-label}
\end{figure}

The next result generalizes \cite[Theorem 3.1, (3)]{7} for even integers with an extra condition.
\begin{theorem}\label{T5}
Let $n \geq2$ be an even integer, $R$ be a commutative ring such that $Z(R)$ is not an ideal of $R$ and suppose there is a $u \in R$ such that $u^n = -1$.	 If the $n-T(Reg(R))$ is connected, then the  $n-T(R)$ is connected.
\end{theorem}
\begin{proof}
Suppose the $n-T(Reg(R))$ is connected. Since $Z(R)$ is not an ideal of $R$, there are $w, v\in Z(R)$ such that $w + v=r\in Reg(R)$. Let $z = u^{-1}v$. Hence, $w + uz = r \in Reg(R)$. Therefore, we have $(w+uz)^n=r^n =\sum^{n}_{k=1}
{n \choose
	k}
u^k w^k z^{n-k}= \sum^{n-1}_{k=1}  {n \choose
	k} u^k w^k z^{n-k}-z^n=wq-z^n$ for some $q\in R$. Therefore, we have $r^n+z^n=wq\in Z(R)$. Since $0, z \in Z(R)$, every vertex of $Z(R)$ is connected to $z$. Since $z$ is connected to $r\in Reg(R)$ and the $n-T(Reg(R))$ is connected, we conclude the $n-T(R)$ is connected.
	\end{proof}
	\begin{example}
The condition $ u^n =-1$ is sufficient but unnecessary for this to hold. Let $R=\mathbb{Z}_{14}$. The $2-T(Reg(R))$) is connected (and thus the $2-T(R)$ is connected). However, we do not have any element $u \in R$ such that $u^2=-1$.
\end{example}
The next result generalizes \cite[Theorem 3.3]{7} for odd integers.
\begin{theorem}\label{T6}
Let  $R$ be a commutative ring such that $Z(R)$ is not an ideal of $R$, and  $n\geq 1$ be an odd integer. The $n-T(R)$ is connected if and only if $R=(z_1,z_2,...,z_m)$ where $z_i\in Z(R)$ for each $1\leq i\leq m$.
\end{theorem}
\begin{proof}
Suppose the $n-T(R)$ is connected. Then there is a path $0- w_1- w_2-  \cdot - w_q- 1$ in $n-T(R)$. This implies that $w_1^n, w_1^n+w_2^n,..., w_{q-1}^n + w_q^n, w_q^n+1 \in Z(R)$. Hence, $1\in (w_1,w_1^n+w_2^n,..., w_{q-1}^n + w_q^n, w_q^n+1)\subseteq (Z(R)) $. Hence, $R=(Z(R))$

Conversely, suppose that $(Z(R))= R$. Then, we must show a path from $0$ to $x$ in $n-T(R)$ for any nonzero $x$ in $R$. By hypothesis, we have $x=v_1+ \cdots +v_q$ for some $v_1,...,v_q\in Z(R)$. Let $w_0=0$ and $w_j=(-1)^{q+j}(v_1+ \cdots +v_j)$ for each integer $1\leq j\leq q$. Since n is odd, we have $(-1)^{n(q+j)}=(-1)^{q+j}$ and hence we have $w_j^n=(-1)^{q+j}(v_1+ \cdots +v_j)^n = w_j^n+w_{j+1}^n=v_{j+1}y \in Z(R)$ for some $y\in R$. Thus, $0- w_1- w_2- \cdots - w_q=x$ is a path from $0$ to $x$ in $n-T(R)$. Let $0\neq a,b\in R$. By the previous argument, there are paths from $a$ to $0$ and from $0$ to $b$. Therefore, there is a path from $a$ to $b$ in $n-T(R)$; hence, $n-T(R)$ is connected.

\end{proof}
\begin{remark}\label{R1}
	In view of the proof of Theorem \ref{T6}, if the $n-T(R)$ is connected for some integer $n\geq 1$, then $R=(z_1,z_2,...,z_m)$, where $z_i\in Z(R)$ for each $1\leq i\leq m$.
	\end{remark}
The next result generalizes \cite[Theorem 3.3]{7} for even integers with an extra condition.
\begin{theorem}
	Suppose $R$ is a commutative ring such that $Z(R)$ is not an ideal of $R$. Let $n$ be an even integer, and suppose there is $u\in R$ such that $u^n=-1$. Then the $n-T(R)$ is connected if and only if $R=(z_1,z_2,...,z_m)$, where $z_i\in Z(R)$ for each $1\leq i\leq m$.
\end{theorem}

\begin{proof}
Suppose the $n-T(R)$ is connected. By Remark \ref{R1}, $R=(z_1,z_2,...,z_m)$, where $z_i\in Z(R)$ for each $1\leq i\leq m$.
Suppose that $(Z(R))=R$. It suffices to show a path from $0$ to $1$ in the $n-T(R)$. By hypothesis, we have $1=v_1+ \cdots +v_q$ for some $v_1,...,v_q\in Z(R)$. Let $w_0=0$ and $w_j= u^{q+j}(v_1+ \cdots +v_j)$ for each integer $1\leq j\leq q$. Since $u^n = -1$, we have $(u)^{n(q+j)}=(-1)^{q+j}$ and hence we have $w_j^n=(-1)^{q+j}(v_1+ \cdots +v_j)^n = v_j^n + v_{j+1}^n= v_{j+1}y \in Z(R)$ for some $y\in R$. Thus, $0- w_1- w_2- \cdots - w_q = 1$ is a path from $0$ to $1$ in the $n-T(R)$. Hence, the $n-T(R)$ is connected.

\end{proof}
The next result generalizes \cite[Theorem 3.4]{7} for odd integers.
\begin{theorem}\label{T6}
Let  $R$ be a commutative ring such that $Z(R)$ is not an ideal of $R$, and $n$ be an odd integer. Suppose that $R=(Z(R))$ (i.e. the $n-T(R)$ is connected). Then \ $diam(n-T(R))=m$, where $m$ is the minimum number of zero divisors $z_1,...,z_m$ such that $(z_1,...z_m)=R$. Moreover, $diam(n-T(R))=d(0,1)$.
\end{theorem}
\begin{proof}
	First, we show that any path from $0$ to $1$ in $R$ has a length of at least $m$. Suppose that $0- b_1- \cdots - b_{k-1}- 1$ is a path from $0$ to $1$ in $n-T(R)$ of length $k$. Then $b_1^n, b_1^n+b_2^n,..., b_{k-1}^n+1\in Z(R)$ and hence $1\in (b_1^n, b_1^n+b_2^n,..., b_{k-1}^n+1)\subseteq (Z(R))$. Thus, $k\geq m$.
	Let $x$ and $y$ be distinct elements in $R$. We will show that there is a path in $R$ with a length less than or equal to $m$. Let $1=z_1+z_2+ \cdot +z_m$ for $z_1,...z_m\in Z(R)$. If $m$ is even, define $z=x+y$. If $m$ is odd, define $z=x-y$. In any case, let $d_0=x$. For each $1\leq k\leq m$ let $d_k=-(x+z(z_1+ \cdot +z_k))$ if $k$ is odd, and $d_k=x+z(z_1+ \cdot +z_k)$ if $k$ is even. Then $d_k^n+d_{k+1}^n=z_{k+1}q \in Z(R)$ for some $q\in R$. Also, $d_m=y$. Therefore, $x- d_1- \cdot - d_m-1- y$ is a path of length at most $m$. In particular, a shortest path between $0$ and $1$ would have length $m$. Therefore, $diam(n-T(R))=m$.
	
\end{proof}
The next result generalizes \cite[Theorem 3.4]{7} for even integers with an extra condition.

\begin{theorem}\label{T7}
Let  $R$ be a commutative ring such that $Z(R)$ is not an ideal of $R$.	and  $n$ be an even integer. Suppose that $R=(Z(R))$ and there exists $u\in R$ such that $u^n=-1$. Then $diam(n-T(R))=m$ where $m$ is the minimum number of zero divisors $z_1,...,z_m$ such that $(z_1,...z_m)=R$. Moreover, $diam(n-T(R))=d(0,1)$.
\end{theorem}
\begin{proof}
First, we show that any path from $0$ to $1$ in $R$ has a length of at least $m$. Suppose that $0- b_1-  \cdot - b_{k-1}- 1$ is a path from $0$ to $1$ in $n-T(R)$ of length $k$. Then $b_1^n, b_1^n+b_2^n,..., b_{k-1}^n+1\in Z(R)$ and hence $1\in (b_1^n, b_1^n+b_2^n,..., b_{k-1}^n+1)\subseteq (Z(R))$. Thus, $k\geq m$.
Let $x$ and $y$ be distinct elements in $R$. We will show that there is a path in $R$ with a length less than or equal to $m$. Let $1=z_1+z_2+ \cdot +z_m$ for $z_1,...z_m\in Z(R)$. If $m$ is even, define $z=x+y$. If $m$ is odd, define $z=x-y$. In any case, let $d_0=x$. For each $1\leq k\leq m$ let $d_k=u(x+z(z_1+ \cdot +z_k))$ if $k$ is odd and $d_k=x+z(z_1+ \cdot +z_k)$ if $k$ is even. Then $d_k^n+d_{k+1}^n=z_{k+1}q \in Z(R)$ for some $q\in R$. Also, $d_m=y$. Therefore, $x- d_1- \cdots - d_m-1- y$ is a path of length at most $m$. In particular, a shortest path between $0$ and $1$ would have length $m$. Therefore, $diam(n-T(R))=m$.

\end{proof}
The following results generalizes \cite[Corollary 3.5]{7}.
\begin{corollary}
Let  $R$ be a commutative ring such that $Z(R)$ is not an ideal of $R$.	if $diam(n-T(R))=m$, then $diam(n-T(Reg(R)))\geq m-2$
\end{corollary}
\begin{proof}
	Since $diam(n-T(R))=d(0,1)=m$, let $0- b_1- \cdots - b_{m-1}- 1$ be a path from $0$ to $1$ in $R$.  Clearly, $b_1\in Z(R)$. Suppose that for some $2\leq i\leq m-1$ we have $b_i$ in $Z(R)$. Therefore, we can have the path $0- b_i- \cdots - 1$ in $n-T(R)$ of length less than $m$. This is a contradiction. Thus, $b_i\in Reg(R)$ for each $2\leq i\leq m-1$ and hence $b_2- \cdots - 1$ is a shortest path between $b_2$ and $1$ in the $n-T(Reg(R))$ of length $m-2$. Therefore $diam(n-T(Reg(R)))\geq m-2$.
\end{proof}
\begin{example}
	Let $R=\mathbb{Z}_{21}$. Every $x\in Reg(R)$, $x$ is not connected to any other element in $R$. Hence, The $2-T(R)$ is not connected.
\end{example}
\begin{example}\label{e333}
	It is not necessary to have $x\in R$ such that $x^n =-1$ for $n-T(R)$ to be connected when $n$ is even. Let $R=\mathbb{Z}_{6}$. There is no $x \in R$ such that $x^2=-1$. However, the $2-T(R)$ is connected as shown in Figure 4. Note that $d(0, 1) = 2$, but $diam(2-T(R)) = 3 \not = d(0, 2)$.
\end{example}
\begin{figure}[!h]
	\centering
	\begin{tikzpicture}
		\tikzset{enclosed/.style={draw, circle, inner sep=0pt, minimum size=.15cm, fill=black}}
		\node[enclosed, label={above, yshift=0cm: 0}] (0) at (1,1) {};
		\node[enclosed, label={above, yshift=0cm: 1}] (1) at (5,0) {};
		\node[enclosed, label={above, yshift=0cm: 2}] (2) at (2,0) {};
		\node[enclosed, label={above, yshift=0cm: 3}] (3) at (4,1) {};
		\node[enclosed, label={above, yshift=0cm: 4}] (4) at (0,0) {};
		\node[enclosed, label={above, yshift=0cm: 5}] (5) at (3,0) {};
		
		\draw (0) -- (3) node[midway, above] (edge1){};
		\draw (0) -- (2) node[midway, above] (edge1){};
		\draw (0) -- (4) node[midway, above] (edge1){};
		\draw (2) -- (4) node[midway, above] (edge1){};
		\draw (3) -- (1) node[midway, above] (edge1){};
		\draw  (1) -- (5) node[midway, above] (edge1){};
		\draw (5)--(3) node[midway, above] (edge1){};
	\end{tikzpicture}
	\caption{$2-T(\mathbb{Z}_6)$}
	\label{fig:enter-label}
\end{figure}
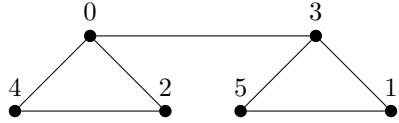

\begin{theorem}
	Let $R=R_1\times \cdot \times R_k$ for some $k \geq 2$ where $R_i$ is a commutative ring with $1\not = 0$ for each $1\leq i \leq k$.  If $n\geq 1$ is an odd integer, then the $n-T(R)$ is connected. Moreover, $diam(n-T(R))=2$.
\end{theorem}
\begin{proof}
		Suppose $x, y \in R$. Then $x=(x_1,...,x_k)$ and $y=(y_1,...,y_k)$. Since $n$ is odd, $x_i^n+(-x_i)^n=0$ (Similarly $y_j^n+(-y_j)^n=0$). Let $r=(r_1,..,-x_i,...,-y_j,...,r_k)$. Then clearly, $x^n+r^n$ is a zero divisor, and $r^n+y^n$ is a zero divisor. Hence, $x- r- y$ is always a path in the $n-T(R)$. Therefore, the $n-T(R)$ is connected and $diam(n-T(R))=2$.
\end{proof}
\begin{theorem}\label{T5}
			Let $n \geq 2$ be an even integer and $R = R_1\times \cdots \times R_k$ for some $k\geq 2$ where $R_i$ is a commutative ring with $1\not = 0$ for each $1\leq i \leq k$. Then the $n-T(R)$ is connected if and only if there is an $x$ in $R_j$ for some $1\leq j\leq k$ such that $x^n +1 \in Z(R_j)$. Furthermore, if there exists a $u \in R$ such that $u^n = -1 \in R$ (i.e., $u^n + 1 = 0 \in Z(R)$), then the $n-T(R)$ is connected and $diam(n-T(R)) = 2$.
\end{theorem}
\begin{proof}
		Suppose the $n-T(R)$ is connected. Then there exists a path from $(0,...,0)$ to $(1,...,1)$ in the $n-T(R)$, say $(0, ... , 0)- \cdots -(x_1, ... , x_k)-(1, ... , 1)$. Hence  $x_i^n+1 \in Z(R_i)$ for some $1 \leq i \leq k$.

	Conversely, suppose there is an $x$ in $R_j$ for some $1\leq j\leq k$ such that $x^n +1 \in Z(R_j)$. Without the loss of generality, we may assume that $j = 1$. Thus, $(0, ... , 0) - (x, 0, ... , 0)-(1, ... , 1)$ is a path of length $2$ in the $n-T(R)$. Hence, the $n-T(R)$ is connected by Theorem \ref{T3}.
	
	Now, suppose that there exists a $u \in R$ such that $u^n = -1 \in R$. Since $R = ((1, 0, \cdots, 0), (0, 1, \cdots , 1))$, the $n-T(R)$ is connected and $diam(n-T(R)) = d((0, ... , 0), (1, ... , 1)) = 2$ by Theorem \ref{T7}.
\end{proof}

\begin{corollary}
	Let $n \geq 2$ be an even integer and $R = R_1\times \cdots \times R_k$ for some $k\geq 2$ where $R_i$ is an integral domain for each $1\leq i \leq k$. Then the $n-T(R)$ is connected if and only if there is an $x$ in $R_j$ for some $1\leq j\leq k$ such that $x^n = -1 \in R_j$ (i.e. $x^n + 1 = 0 \in Z(R_j))$. Furthermore, if the $n-T(R)$ is connected and there exists a $u \in U(R)$ such that $u^n = -1 \in R$ (i.e., $u^n + 1 = 0 \in Z(R)$), then $diam(n-T(R)) = 2$.
	\end{corollary}

\begin{proof}
This follows directly from Theorem \ref{T5} and the fact that $R_i$ has no zero divisors except for $0$ for each $1 \leq i \leq k$.
\end{proof}
See Example \ref{e333}, $R = Z_{6}$ is ring-isomorphic to $A = Z_2\times Z_3$. Note that $d((0, 0), (1, 1)) = 2$ in the $2-T(A)$, but $diam(2-T(A))= 3$.
\begin{example}\label{E1}
$R=\mathbb{Z}_3 \times \mathbb{Z}_3$. There is no $x\in R$ such that $x^2=-1$, and the $2-T(R)$ is not connected. See figure 5. However, the $2-T(\mathbb{Z}_{14})$ is connected, and there is no $x \in \mathbb{Z}_{14}$ such that $x^2 = -1$.
\end{example}

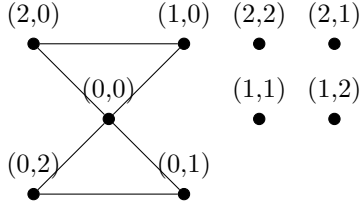
\begin{figure}[!h]
	
	\centering
	\begin{tikzpicture}
		\tikzset{enclosed/.style={draw, circle, inner sep=0pt, minimum size=.15cm, fill=black}}
		
		\node[enclosed, label={above, yshift=0cm: (0,0)}] (0) at (2,1) {};
		\node[enclosed, label={above, yshift=0cm: (0,1)}] (1) at (3,0) {};
		\node[enclosed, label={above, yshift=0cm: (0,2)}] (2) at (1,0) {};
		\node[enclosed, label={above, yshift=0cm: (1,0)}] (3) at (3,2) {};
		\node[enclosed, label={above, yshift=0cm: (2,0)}] (4) at (1,2) {};
		\node[enclosed, label={above, yshift=0cm: (2,2)}] (5) at (4,2) {};
		\node[enclosed, label={above, yshift=0cm: (2,1)}] (5) at (5,2) {};
		\node[enclosed, label={above, yshift=0cm: (1,1)}] (5) at (4,1) {};
		\node[enclosed, label={above, yshift=0cm: (1,2)}] (5) at (5,1) {};
		
		\draw (0) -- (1) node[midway, above] (edge1) {} ;
		\draw (0) -- (2) node[midway, above] (edge1) {} ;
		\draw (1) -- (2) node[midway, above] (edge1) {} ;
		\draw (0) -- (3) node[midway, above] (edge1) {} ;
		\draw (0) -- (4) node[midway, above] (edge1) {} ;
		\draw (3) -- (4) node[midway, above] (edge1) {} ;

	\end{tikzpicture}
	\caption{$2-T(\mathbb{Z}_3\times \mathbb{Z}_3)$}
	
\end{figure}
In view of Example \ref{E1}, we have the following result.
\begin{corollary}
	 Let $n$ be an even integer and $R=R_1\times \cdot \times R_k$ for some $k\geq 2$ where $R_i$ is an integral domain for each $1\leq i\leq k$. Then the $n-T(R)$ is connected if and only if there is an $x$ in $R_j$ for some $1\leq j\leq k$ such that $x^n=-1$ in $R_j$ (i.e., $x^n + 1 = 0 \in Z(R)$).
\end{corollary}

\begin{proof}
	This follows directly from Theorem \ref{T5} and the fact that $R_i$ has no zero divisors except for $0$ for each $1\leq i\leq k$.
\end{proof}
\begin{example}
The $2-T(\mathbb{Z}_5\times\mathbb{Z}_3)$ is connected, since in $\mathbb{Z}_5$ we have $2^2=4=-1\in \mathbb{Z}_5$.
\end{example}

\begin{example}
The	$2-T(\mathbb{Z}_2\times\mathbb{Z}_2)$ is connected, since in $\mathbb{Z}_2$ we have $1^2=1=-1\in \mathbb{Z}_2$. This is shown in Figure 6.
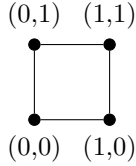
\begin{figure}[!h]
	\centering
	\begin{tikzpicture}
		\tikzset{enclosed/.style={draw, circle, inner sep=0pt, minimum size=.15cm, fill=black}}
		
		\node[enclosed, label={below, yshift=0cm: (0,0)}] (0) at (0,0) {};
		\node[enclosed, label={below, yshift=0cm: (1,0)}] (1) at (1,0) {};
		\node[enclosed, label={above, yshift=0cm: (0,1)}] (2) at (0,1) {};
		\node[enclosed, label={above, yshift=0cm: (1,1)}] (3) at (1,1) {};
		
		\draw (0) -- (1) node[midway, below] (edge1) {} ;
		\draw (0) -- (2) node[midway, below] (edge1) {} ;
		\draw (1) -- (3) node[midway, above] (edge1) {} ;
		\draw (2) -- (3) node[midway, above] (edge1) {} ;

	\end{tikzpicture}
	\caption{$2-T(\mathbb{Z}_2\times \mathbb{Z}_2)$}
\end{figure}
\end{example}
In the following result, we generalize \cite[Example 3.8]{7} but our proof differs entirely from that in \cite{7}
\begin{theorem}
	For any integer $m \geq 2$, there exists a ring $R$ such that for every integer $n\geq 1$, the  $n-T(R)$ is connected and $diam(n-T(R))=m$.
\end{theorem}
\begin{proof}
Let $n \geq 1$, $A=\mathbb{Z}_2[X_1,X_2,...,X_{m-1},W_1,W_2,...,W_m]$,

$I=(W_1X_1,...,W_{m-1}X_{m-1}, W_m(X_1+ \cdot +X_{m-1}+1),\{W_iW_j| 1\leq i<j\leq m\})$, and $R=A/I$. Denote $x_i=X_i+I\in R$ and $w_i=W_i+I\in R$. Then $(-1)^n=1$. We can construct the following path from $0$ to $1$ in the $n-T(R)$
	\begin{equation}
		0- x_1- (x_1+x_2)- \cdots - (x_1+x_2+ \cdot +x_{m-1})- 1
	\end{equation}
	Note that this is the shortest path from $0$ to $1$ of length $m$ in the $n-T(R)$ since any sum of the $x_i$'s is not a zero divisor in $R$, and thus we can not make a shorter path. Therefore, by Theorem \ref{T6} (In the case of odd $n$) and Theorem \ref{T7} (In the case of even $n$), $diam(n-T(R))=m$.
\end{proof}
{\bf Acknowledgment}

  The authors thank the referee for helpful comments.

\bibliographystyle{amsalpha}

\end{document}